\documentclass[a4paper,american,10pt]{amsart}

\usepackage{amsmath,amsthm,amssymb}
\usepackage{hyperref}
\usepackage[american]{babel}

\newtheorem{lemma}{Lemma}
\newtheorem{theorem}{Theorem}
\newtheorem{remark}{Remark}
\newtheorem{conjecture}{Conjecture}

\renewcommand\Im{\mathfrak{Im}(z)} 
\newcommand\ZZ{\mathbb Z}
\newcommand\RR{\mathbb R}
\newcommand\CC{\mathbb C}
\newcommand\NN{\mathbb N}
\newcommand\vect[2]{\tiny
\begin{pmatrix}
#1\\
#2
\end{pmatrix}
}

\begin{document}

\author{Nikita Kalinin}

\title[Evaluating lattice sums via telescoping]{Evaluating lattice sums via telescoping on $SL_+(2,\ZZ)$: a short proof of
$\displaystyle \sum%_{(x,y)\in SL_+(2,\ZZ)} 
\frac{1}{\|x\|^2\|y\|^2\|x+y\|^2}=\frac{\pi}{4}$ and of Zagier's identity}

\address{Guangdong Technion Israel Institute of Technology (GTIIT),
241 Daxue Road, Shantou, Guangdong Province 515063, P.R.~China}
\address{Technion -- Israel Institute of Technology, Haifa 3200003, Israel}
\email{nikaanspb@gmail.com}

\maketitle

\begin{abstract}
We study lattice sums $\sum \frac{1}{(\|x\|\|y\|\|x+y\|)^s}$ taken over $SL_+(2,\ZZ)$, i.e.\ the set of pairs $(x,y)$ of primitive lattice vectors in $\ZZ_{\geq 0}^2$ with $\det(x, y) = 1$.  We prove convergence of these and similar (determinant weighted) sums and introduce a new telescoping method on $SL_+(2,\ZZ)$ that yields, in particular, $$\sum_{(x,y)\in SL_+(2,\ZZ)} \frac{1}{\|x\|^2\,\|y\|^2\,\|x+y\|^2}=\frac{\pi}{4},$$ and a short proof of Zagier's identity $D_{1,1,1}=2E(z,3)+\pi^3\zeta(3)$. 


\end{abstract}

%\subjclass[2020]{11M41, 11P21, 33B15, 11Y60, 11F03}
%
%\keywords{lattice sums, telescoping series, lattice geometry, modular graph functions}

\section{Introduction}

The most basic example of a series over a lattice (here $\ZZ$) is the Riemann zeta function
\[
\zeta(s)=\sum_{n=1}^\infty \frac{1}{n^s},
\]
which converges for $\Re(s)>1$.
A closely related family is given by the holomorphic Eisenstein series
\[
G_{2k}(z)
=\frac{1}{2}\sum_{(m,n)\in\ZZ^2\setminus \{(0,0)\}}
\frac{1}{(m+nz)^{2k}},
\qquad z\in\mathcal{H},\ k\ge2.
\]

A similar sum, which is the main object of this paper, is defined as follows. Each {\it primitive} (i.e.\ with coprime coordinates) vector $v$ in $\ZZ_{>0}^2$ can be uniquely presented as a sum of two vectors $x(v),y(v)\in\ZZ_{\geq 0}^2$ spanning a parallelogram of oriented area one, i.e.\ $\det(x(v),y(v))=1$. Define

\[
S =\sum_{\text{primitive } v\in \ZZ_{>0}^2}\frac{1}{\|v\|\,\|x(v)\|\,\|y(v)\|}=\sum_{(x, y) \in SL_+(2,\ZZ)} \frac{1}{\|x\|\,\|y\|\,\|x + y\|}, 
\]
where  $$SL_+(2,\ZZ) = \{(x, y)\mid x,y \in \ZZ_{\geq 0}^2,\ \det(x, y) = 1\}=$$
$$=\left\{\begin{pmatrix} a & b \\ c & d \end{pmatrix} \,\middle|\, a,b,c,d\in\ZZ_{\geq 0},\ ad-bc=1\right\}
.$$

Although the book \cite{borwein2013lattice} is entirely devoted to lattice sums, it does not mention $S$, and
to the best of our knowledge this particular sum has not been studied before.

Direct numerical computations show that the series $S$  converges, albeit slowly. For example, the number of terms whose contributions lie in successive ranges is as follows:
\begin{itemize}
  \item 5 terms contribute values in $(10^{-1},1]$,
  \item 30 terms in $(10^{-2},10^{-1}]$,
  \item 158 terms in $(10^{-3},10^{-2}]$,
  \item 790, 3874, 18512, 88006, 414680, 1943594, 9083386, 42354734, 197214206, 917339530, etc., terms in the subsequent intervals $(10^{-(n+1)},10^{-n}]$ for $n=3,\dots,12$.
\end{itemize}

Each successive count is approximately five times the previous one, 
suggesting that the contribution of terms in the range $(10^{-(n+1)},10^{-n}]$ 
to the total sum is roughly $2^{-n}$. Hence, summing all terms with absolute value larger than $10^{-15}$ should yield a precision 
of about $2^{-15} \approx 3\cdot 10^{-5}$ (according to the above heuristic), which is rather poor. 
In practice, we obtain only three correct decimal digits in
\[
S=\sum_{(x, y) \in SL_+(2,\ZZ)} \frac{1}{\|x\|\,\|y\|\,\|x + y\|}=3.432\ldots
\]
after several days of computation. It would be very interesting to determine the exact value of $S$, that is, to express it in closed form 
in terms of standard special values such as $\zeta(s)$, $E(z,s)$, and related functions.

Sums of the form
\[
\sum_{(x,y)\in\ZZ^2\setminus\{(0,0)\}} 
\frac{1}{\|x\|^k\,\|y\|^n\,\|x+y\|^m}
\]
appear in physics under the name ``modular graph functions'' (MGFs) in perturbative string theory 
\cite{d2015modular,broedel2019elliptic,gerken2020generating}. 
The identity of Zagier discussed below is the simplest nontrivial example of such an MGF.

\subsection{Telescoping over $SL_+(2,\ZZ)$.}
We prove the convergence of $S$ and accelerate its numerical evaluation by using a telescoping argument for a suitable auxiliary series. Note that any basis
\[
p=(x,y)\in SL_+(2,\ZZ)
\]
generates two other bases
\[
p_x=(x,x+y),\qquad p_y=(x+y,y),
\]
and every element of $SL_+(2,\ZZ)$ is obtained from the seed basis
\[
(e_1,e_2)\in SL_+(2,\ZZ),\qquad e_1=\vect{1}{0},\ e_2=\vect{0}{1},
\]
by iterating the operations $p\mapsto p_x$ and $p\mapsto p_y$ (cf.\ the Stern--Brocot tree). Hence, to compute a sum $\sum_{p\in SL_+(2,\ZZ)}f(p)$ one can use the {\bf telescoping argument  on $SL_+(2,\ZZ)$}: 
we look for a function $F$ on $SL_+(2,\ZZ)$ such that
\[
f(p)=F(p)-F(p_x)-F(p_y)\qquad\text{for all }p\in SL_+(2,\ZZ).
\]
For any finite subtree $T$ of the rooted tree generated from $(e_1,e_2)$, summing this identity over all interior vertices $p\in T$ yields a complete cancellation of intermediate terms, leaving only the contribution from the root and the boundary $\partial T$. If $F(p)\to0$ sufficiently fast as $p$ tends to infinity in $SL_+(2,\ZZ)$, the boundary contribution tends to zero as $T$ exhausts $SL_+(2,\ZZ)$, and we obtain

\begin{align*}
\sum_{p\in SL_+(2,\ZZ)}f(p)
&= F(e_1,e_2)-F(e_1,e_1+e_2)-F(e_1+e_2,e_2)\\
&\quad+F(e_1,e_1+e_2)-F(e_1,2e_1+e_2)-F(2e_1+e_2,e_1+e_2)\\
&\quad+F(e_1+e_2,e_2)-F(e_1+2e_2,e_2)-F(e_1+e_2,e_1+2e_2)+\cdots\\
&= F(e_1,e_2),
\end{align*}
where the dots indicate further cancellations along the branches of the tree.

One can also formulate this telescoping procedure in terms of Farey pairs, see for example \cite{hata1995farey}. 
In some situations $F$ does tend to zero, but not fast enough for the boundary contribution to vanish; for instance,
\[
F\!\left(\vect{a}{b},\vect{c}{d}\right)\sim \frac{1}{bd}.
\]
In such cases the limit of the boundary contribution can be computed as an explicit integral of a ``charge at infinity'', see \cite{certain}.

Simultaneously, the same telescoping method proves the convergence of the angular series
\[
\sum_{\substack{m,n\in \ZZ_{>0}\\ \gcd(m,n)=1}} 
\frac{\alpha_{m,n}}{\sqrt{m^2+n^2}},
\]
where, for each primitive vector $v=\vect{m}{n}$, we choose vectors $x(v),y(v)\in\ZZ_{\ge0}^2$ with $x(v)+y(v)=v$ and $\det(x(v),y(v))=1$, and denote by $\alpha_{m,n}$ the angle between $x(v)$ and $y(v)$.

\subsection{Zagier's formula.}
In an unpublished\footnote{I am in possession of an unfinished note by Zagier, titled ``Evaluation of lattice sums'', which is likely the same note.} note \cite{zagiernotes}, Zagier considers the following sum. Let $z=x+iy\in\mathcal H =\{z\in\CC| \Im>0\}$ and
\[
D_{1,1,1}(z)
=\sideset{}{'}\sum_{\substack{\omega_1+\omega_2+\omega_3=0\\ \omega_k\in \ZZ z+\ZZ}}
\frac{y^{3}}{|\omega_1\omega_2\omega_3|^{2}}.
\]
Here and throughout, $\sideset{}{'}\sum$ denotes a sum with the terms corresponding to infinite summands omitted. In the present case, this means that we remove all terms for which at least one of the $\omega_i$ is equal to $0$.

Then, as shown in \cite{zagiernotes,d2017modular,brown2018class},
\begin{equation}
\label{eq_2}
D_{1,1,1}(z) = 2E(z,3)+\pi^3\zeta(3),
\end{equation}
where
\[
E(z,s)
= \frac{1}{2}\sideset{}{'}\sum_{m,n\in\ZZ}\frac{y^s}{|mz+n|^{2s}}
\]
is the (real-analytic) non-holomorphic Eisenstein series.

Using our telescoping method, we prove that
\begin{equation}
\label{eq_za}
\sideset{}{'}\sum_{\substack{\omega_1+\omega_2+\omega_3=0\\ \omega_k\in \ZZ z+\ZZ}}
\frac{|\det (\omega_1,  \omega_2)|^{-s}}{|\omega_1\omega_2\omega_3|^{2}}
= \frac{6\pi}{y^3}\,\zeta(s+3)\,\zeta(s+2).
\end{equation}
By substituting $s=0$ and separating the contribution of collinear triples $(\omega_1,\omega_2,\omega_3)$, we recover Zagier's formula \eqref{eq_2} for general $z$; see Section~\ref{sec_4} for details. It seems plausible that other modular sums considered in \cite{d2017modular,brown2018class} can also be evaluated using our telescoping method.

\subsection*{Outline of the paper}

In Section~\ref{sec_2} we use the telescoping method to analyze the convergence properties of the lattice sum 
\[
S=\sum_{(x, y) \in SL_+(2,\ZZ)} \frac{1}{\|x\|\,\|y\|\,\|x + y\|}.
\]
Several related identities and estimates are also established there.

Section~\ref{sec_tele} extends these ideas and uses an \emph{angle-partition argument} to evaluate the identity
\[
\sum_{(x, y) \in SL_+(2,\ZZ)} \frac{1}{\|x\|^2\,\|y\|^2\,\|x + y\|^2} =\frac{\pi}{4},
\]
which goes back to Hurwitz \cite{hurwitz1905darstellung} and has been revisited in more recent works
\cite{duke2021class,o2024topographs,d2015modular,d2017modular}. In Section~\ref{sec_4} we prove \eqref{eq_za}, providing an alternative proof of Zagier's formula \eqref{eq_2}. %We also derive a more general identity for determinant-weighted lattice sums of the form \eqref{eq_za} over the upper half-plane in Section~\ref{sec_5}. 
The following conjecture looks plausible.

\begin{conjecture}
The function
\[
S(s)=\sum_{(x, y) \in SL_+(2,\ZZ)} \frac{1}{(\|x\|\,\|y\|\,\|x + y\|)^s}
\]
admits a meromorphic continuation to the half-plane $\Re(s)>2/3-\varepsilon$ for some $\varepsilon>0$ and has a simple pole at $s=2/3$.
\end{conjecture}

In Section~\ref{sec_5} we consider generalizations of the preceding sums by inserting a determinant power factor and, for $s>1$, we show convergence of
\[
S^{\mathrm{det}}(s,z)
=\sideset{}{'}\sum_{\substack{w_1 + w_2 + w_3 = 0 \\ w_i \in \ZZ z+\ZZ}}
\frac{|\det(w_1, w_2)|^{-s}}{|w_1|\,|w_2|\,|w_3|}.
\]

It would be interesting to further investigate the analytic properties of $S^{\mathrm{det}}(s,z)$, in particular at the cusp $z\to i\infty$ (cf.\ \cite{d2015modular,broedel2019elliptic,gerken2020generating}).

\subsection{Discussion.}
One can view the set of all bases in $\ZZ_{\geq 0}^2$ (that is, $SL_+(2,\ZZ)$) as a graph whose vertices are bases and whose edges correspond to the elementary moves $p \mapsto p_x=(x,x+y)$ and $p \mapsto p_y=(x+y,y)$. The telescoping argument suggests using ideas from graph theory or homological algebra: one could try to characterize which summands $f$ admit such an $F$ (exactness of a $1$-cocycle) and what conditions on growth at infinity are needed.

By exploring these tools, one might uncover new classes of summation identities or even new invariants that arise as “charges at infinity” when telescoping fails (some examples are presented in Section~\ref{sec_tele} and in \cite{certain}). In summary, systematizing the telescoping method could transform it from an ad hoc trick into a broadly applicable technique in analytic number theory, much as the Abel summation method and the Euler–Maclaurin formula.

\section{The sum $\sum\frac{1}{\|x\|\,\|y\|\,\|x+y\|}$ and related sums}\label{sec_2}

Define the function $G(x,y)$ for $x,y\in \ZZ^2$ by
\begin{equation}
\label{eq_G}
G(x, y) = \|x\| - \frac{x \cdot y}{\|y\|}.
\end{equation}
We have
\[
G(x,y)= \|x\|\,(1-\cos\alpha)
= \|x\|\,\frac{\alpha^2+O(\alpha^4)}{2}
= \frac{D\,(\alpha+ O(\alpha^3))}{2\|y\|},
\]
where $\alpha$ is the angle between $x$ and $y$, and
\[
D = \det(x,y) = \|x\|\,\|y\|\,\sin\alpha.
\]
Here and throughout, $O(\alpha^k)$ denotes a function of $\alpha$ whose absolute value is bounded by $C|\alpha|^k$ for some absolute constant $C$ and all sufficiently small $\alpha$. Let $\alpha_x$ be the angle between $x$ and $x+y$, and $\alpha_y$ the angle between $y$ and $x+y$. Then one observes the identity
\begin{align*}
(**)\quad
G(x, y) &- G(x + y, y) - G(x, x + y) \\
&= \left( \|x\| - \frac{x \cdot y}{\|y\|} \right)
  - \left( \|x+y\| - \frac{(x+y) \cdot y}{\|y\|} \right)
  - \left( \|x\| - \frac{x \cdot (x+y)}{\|x+y\|} \right)\\
&=\|y\| -\|x+y\| +\frac{x \cdot (x+y)}{\|x+y\|}=\\
&= \frac{1}{\|x+y\|}\left(\|y\|\,\|x+y\|-y\cdot(x+y)\right)=\|y\|\,(1-\cos \alpha_y). 
\end{align*}

\begin{theorem}
\label{thm_xyz} The following sum  converges.
$$\sum_{(x, y) \in SL_+(2,\ZZ)} \frac{1}{\|x\|\,\|y\|\,\|x+y\|}$$
\end{theorem}

\begin{proof}
Since we sum over $(x,y)\in SL_+(2,\ZZ)$ with $\det(x,y)=1$, the angle $\alpha$ between $x$ and $y$ tends to $0$ as $\|x\|,\|y\|\to\infty$. Thus the Taylor expansions in $\alpha$, $\alpha_x$ and $\alpha_y$ used below are valid for all but finitely many pairs, and these exceptional pairs do not affect convergence.

If the area of the parallelogram spanned by $x$ and $y$ is $D$, then the parallelogram spanned by $y$ and $x+y$ also has area $D$, so
\[
\|y\|\,\|x+y\|\sin\alpha_y
= D
= (\alpha_y+O(\alpha_y^3))\,\|y\|\,\|x+y\|
= D\,(1+O(\alpha_y^2)),
\]
and we have
\begin{equation}\label{eq4}
\|y\|\,(1-\cos \alpha_y)
= \frac{\|y\|\,(\alpha_y^2+O(\alpha_y^4))}{2}
= \frac{D\bigl(\alpha_y+O(\alpha_y^3)\bigr)}{2\|x+y\|}.
\end{equation}

On the other hand, for the angle $\alpha$ between $x$ and $y$, we have
\begin{equation}\label{eq5}
\frac{1}{\|x\|\,\|y\|\,\|x+y\|} = \frac{\sin \alpha}{D \|x+y\|} = \frac{\alpha+O(\alpha^3)}{D \|x+y\|}.
\end{equation}

Define
\[
F(x, y)
= 2\left(\|x\| - \frac{x \cdot y}{\|y\|}\right)
 +2\left(\|y\| - \frac{x \cdot y}{\|x\|}\right)
= 2\bigl(G(x,y)+G(y,x)\bigr).
\]

Using $(**)$ first for the pair $(x,y)$ and then for the pair $(y,x)$, and applying \eqref{eq4} in each case, we obtain
\begin{align*}
&G(x, y) - G(x + y, y) - G(x, x + y) \\
&\quad+ G(y, x) - G(y,x+y) - G(x+y, x) \\
&\qquad= \frac{D\bigl(\alpha_y+O(\alpha_y^3)\bigr)}{2\|x+y\|}
      + \frac{D\bigl(\alpha_x+O(\alpha_x^3)\bigr)}{2\|x+y\|}.
\end{align*}

Since the angle $\alpha$ between $x$ and $y$ is the sum of the angles $\alpha_x$ and $\alpha_y$ between $x,x+y$ and $x+y,y$, respectively, we have
\[
\alpha = \alpha_x+\alpha_y,\qquad
\alpha_x^3+\alpha_y^3 = O(\alpha^3).
\]
Comparing the last display with \eqref{eq5} and using the definition of $F$, we obtain
\[
F(x, y) - F(x + y, y) - F(x, x + y)
= \frac{D^2\bigl(1+O(\alpha^2)\bigr)}{\|x\|\,\|y\|\,\|x+y\|}.
\]
Since $D=\det(x,y)=1$ for $(x,y)\in SL_+(2,\ZZ)$, the preceding identity gives
\[
F(x, y) - F(x + y, y) - F(x , x + y)
= \frac{1+O(\alpha^2)}{\|x\|\,\|y\|\,\|x+y\|}.
\]
By applying this identity over an appropriate finite domain (for instance, lattice pairs $(x, y) \in SL_+(2,\ZZ)\cap [0,n]^4$) and using the cancellation in the telescoping sum, we obtain
\[
\sum_{(x, y) \in SL_+(2,\ZZ) \cap [0, n]^4}
\bigl(F(x, y) - F(x + y, y) - F(x , x + y)\bigr)
\longrightarrow F\left(\vect{1}{0}, \vect{0}{1}\right),
\]
as $n\to \infty$, because the contributions from the boundary of the domain tend to zero. Indeed, by \eqref{eq_G} and the expansion above, each boundary term is $O\!\left(\frac{\alpha_i}{|y_i|}\right)$; the total sum of the angles $\alpha_i$ along a fixed ray is $\pi/2$, so the sum of the terms with $|y_i|<N$ tends to zero for fixed $N$, while the contribution of the remaining terms is bounded by $C\,\pi/N$.

Denote $e_1 = \vect{1}{0}$ and $e_2 = \vect{0}{1}$. Since $G(e_1,e_2)=G(e_2,e_1)=1$, we have
\[
F(e_1, e_2) = 2\bigl(G(e_1,e_2)+G(e_2,e_1)\bigr) = 4.
\]
Moreover, for all but finitely many pairs $(x,y)\in SL_+(2,\ZZ)$ we have $\alpha$ sufficiently small, so that $1+O(\alpha^2)$ is bounded away from $0$ and infinity. Hence the summand
\[
\frac{1}{\|x\|\,\|y\|\,\|x+y\|}
\]
is comparable (up to a bounded multiplicative factor) to the telescoping summand
\[
F(x, y) - F(x + y, y) - F(x , x + y),
\]
whose series converges to $F(e_1,e_2)$. This comparison implies the convergence of the original series, which completes the proof.
\end{proof}

\begin{remark}
Recall that $\alpha=\alpha_x+\alpha_y$ and $\sin\alpha = 1/(\|x\|\,\|y\|)$ for primitive pairs. Computing instead the sum of the differences
\begin{align*}
&F(x, y) - F(x + y, y) - F(x, x + y) - \frac{1}{\|x\|\,\|y\|\,\|x+y\|} \\
&\quad=\frac{1}{\|x+y\|}
\left(\frac{2-2\cos \alpha_y}{\sin \alpha_y}
      +\frac{2-2\cos \alpha_x}{\sin \alpha_x}
      - \sin(\alpha_y+\alpha_x)\right) \\
&\quad=\frac{1}{\|x+y\|}
\left(2\Bigl(\alpha_x+\frac{\alpha_x^3}{12}
             +\alpha_y+\frac{\alpha_y^3}{12}\Bigr)
      -\Bigl(\alpha_x+\alpha_y - \frac{(\alpha_x+\alpha_y)^3}{6}\Bigr)
      + O(\alpha^5)\right) \\
&\quad < \frac{\alpha^3}{\|x+y\|}
       < \frac{8}{\|x\|^3\,\|y\|^3\,\|x+y\|},
\end{align*}
we obtain a much faster convergent series. In particular, summing this difference and then adding the telescopic value $F(e_1,e_2)=4$ yields
\[
\sum_{(x, y) \in SL_+(2,\ZZ)} \frac{1}{\|x\|\,\|y\|\,\|x+y\|}
=3.43232752714\ldots
\]
to high precision within the same computational time.
\end{remark}

Note that the sum
\[
\sum_{(x,y)\in SL_+(2,\ZZ)} \frac{1}{\|x\|^3\,\|y\|^3\,\|x+y\|}
\]
converges. Indeed, from the triangle inequality we have
\[
\|x+y\|\le \|x\|+\|y\|\le 2\max(\|x\|,\|y\|),
\]
so
\[
\frac{1}{\|x\|^3\,\|y\|^3\,\|x+y\|}
\leq \frac{8}{\|x+y\|^4},
\]
and the sum of $1/\|x+y\|^4$ over primitive vectors $x+y\in\ZZ^2$ converges.

\begin{remark}
To accelerate the computation of the sum
\[
\sum_{(x,y)\in SL_+(2,\ZZ)}\frac{1}{\|x\|\,\|y\|\,\|x+y\|^2},
\]
one can use a telescoping argument for
\[
1-\frac{x\cdot y}{\|x\|\,\|y\|} = 1-\cos\alpha,
\]
where $\alpha$ is the angle between $x$ and $y$. Indeed, since
\[
\frac{1}{\|x\|\,\|y\|\,\|x+y\|^2}
= \sin\alpha_y\,\sin\alpha_x
= \alpha_y\alpha_x +O(\alpha^3),
\]
and
\[
\alpha_y\alpha_x
= \bigl(1-\cos(\alpha_y+\alpha_x)\bigr)
  -\bigl(1-\cos\alpha_y\bigr)
  -\bigl(1-\cos\alpha_x\bigr)
  + O(\alpha^3),
\]
we obtain a telescoping representation up to an absolutely convergent error term, analogous to the case treated in Theorem~\ref{thm_xyz}.
\end{remark}

\begin{theorem}
For every primitive vector $(m,n) \in \ZZ_{\geq 0}^2$ there exists a unique pair $(x,y)\in SL_+(2,\ZZ)$ such that $x+y=(m,n)$. Let $\alpha_{m,n}$ denote the angle between $x$ and $y$. Then the following sum converges:
\[
\sum_{\substack{(m,n)\in\ZZ_{\geq 0}^2\\ \gcd(m,n)=1}} \frac{\alpha_{m,n}}{\sqrt{m^2+n^2}}.
\]
\end{theorem}

\begin{proof}
For $(x,y)\in SL_+(2,\ZZ)$ we have $\det(x,y)=1$, so
\[
\sin\alpha_{m,n} = \frac{1}{\|x\|\,\|y\|}.
\]
For all but finitely many primitive vectors $(m,n)$ the angle $\alpha_{m,n}$ is small, and we may write
\[
\alpha_{m,n} + O(\alpha_{m,n}^3)
= \sin\alpha_{m,n}
= \frac{1}{\|x\|\,\|y\|}.
\]
Since $\sqrt{m^2+n^2}=\|x+y\|$, this gives
\[
\frac{\alpha_{m,n}}{\sqrt{m^2+n^2}}
= \frac{1}{\|x\|\,\|y\|\,\|x+y\|}
  + \frac{O(\alpha_{m,n}^3)}{\|x+y\|}.
\]
The main term is exactly the summand from Theorem~\ref{thm_xyz}, and the error term is absolutely summable by the same argument as in the proof of that theorem. Hence the series converges.
\end{proof}

For completeness we also prove here the following theorem.

\begin{theorem}\label{th_5}
The series
\[
\sum_{(x, y) \in SL_+(2,\ZZ)} \frac{1}{\|x\|^{s}\,\|y\|^{s}\,\|x+y\|^{s}}
\]
converges for $s>2/3$ and diverges for $s\leq 2/3$.
\end{theorem}

\begin{proof}
Fix a primitive vector $x\in\ZZ_{\ge0}^2$. All $y$ with $(x,y)\in SL_+(2,\ZZ)$ are of the form
\[
y = y_0 + n x,\qquad n\in\ZZ_{\ge0},
\]
for some fixed $y_0$ with $\det(x,y_0)=1$: indeed, if $\det(x,y)=\det(x,y_0)=1$ then $\det(x,y-y_0)=0$, so $y-y_0$ is a multiple of $x$. Restricting to the nonnegative cone gives $n\ge0$.

For large $n$ we have
\[
(n+1)\|x\|\geq\|y\|\geq n\|x\|,\qquad (n+2)\|x\|\geq \|x+y\|\geq (n+1)\|x\|,
\]
and hence  
\[
\frac{1}{\|x\|^{3s}}\cdot \frac{1}{(n+2)^{2s}}\leq \frac{1}{\|x\|^{s}\,\|y\|^{s}\,\|x+y\|^{s}}
\leq  \frac{1}{\|x\|^{3s}}\cdot \frac{1}{n^{2s}}.
\]
Since $\sum_{n\ge1}n^{-2s}$ converges for $s>1/2$, the convergence of 
\[
\sum_{n\ge1} \frac{1}{\|x\|^{s}\,\|y\|^{s}\,\|x+y\|^{s}}\]
is equivalent to the convergence of 
\(
\frac{1}{\|x\|^{3s}},
\)
which converges if and only if
\[
\int_1^\infty R^{2}\cdot R^{-3s}\,\frac{dR}{R}
= \int_1^\infty R^{1-3s}\,dR
\]
converges, which holds exactly when $1-3s<-1$, i.e.\ $s>2/3$.
\end{proof}

\section{Telescoping}
\label{sec_tele}

Let $\det(x, y)$ denote the determinant of the $2\times 2$ matrix with columns $x,y\in\ZZ^2$. Recall
\[
SL_+(2,\ZZ) = \big\{ (x, y) \mid  x,y\in \ZZ_{\geq 0}^2,\ \det(x, y) = 1 \big\},
\]
that is, the set of pairs of lattice vectors $x=(a,b),y=(c,d)$ in the first quadrant that span lattice parallelograms of oriented area one, i.e.\ $ad-bc=1$.

\begin{theorem} 
\label{th_one}
\[
4\sum_{(x,y)\in SL_+(2,\ZZ)} \frac{1}{\|x\|^2\,\|y\|^2\,\|x+y\|^2} = \pi.
\]
\end{theorem}

\begin{proof}
Define
\[
P(x,y) = \frac{x\cdot y}{\|x\|^2\,\|y\|^2},\qquad P:(\ZZ^2)^2\to \RR.
\]

An explicit computation shows that
\begin{equation}
\label{eq_1}
P(x,y)-P(x+y,y)-P(x,x+y)
= \frac{-2\,\det(x,y)^2}{\|x\|^2\,\|y\|^2\,\|x+y\|^2}.
\end{equation}

Fix $n\in\NN$ and consider the finite set
\[
A_n=\bigl\{(x,y)\in SL_+(2,\ZZ)\mid x,y\in[0,n]^2\bigr\}.
\]
We sum the expression
\[
P(x,y)-P(x+y,y)-P(x,x+y)
\]
over $(x,y)\in A_n$. By cancelling identical terms with opposite signs as in Section~\ref{sec_2}, we obtain
\[
\sum_{(x,y)\in A_n} \bigl(P(x,y)-P(x+y,y)-P(x,x+y)\bigr)
= P\left(e_1,e_2\right) + S_n,
\]
where $S_n$ is the sum of the boundary contributions
\[
S_n = \sum_{(x,y)\in B_n} \bigl(-P(x+y,y)-P(x,x+y)\bigr),
\]
and
\[
B_n=\bigl\{(x,y)\in A_n\mid x+y\notin[0,n]^2\bigr\}.
\]

Each element $(x,y)\in B_n$ represents a parallelogram spanned by $x$ and $y$. All these parallelograms have area one, and their angles at the origin partition the right angle $\pi/2$ of the first quadrant.

Let $\alpha$ be the angle between $x$ and $y$. For $(x,y)\in SL_+(2,\ZZ)$ we have
\[
\det(x,y)=\|x\|\,\|y\|\,\sin\alpha =1,
\]
so $\sin\alpha = 1/(\|x\|\,\|y\|)$. Hence
\[
P(x,y)
= \frac{x\cdot y}{\|x\|^2\,\|y\|^2}
= \frac{\|x\|\,\|y\|\,\cos\alpha}{\|x\|^2\,\|y\|^2}
= \frac{\cos\alpha}{\|x\|\,\|y\|}
= \cos\alpha\,\sin\alpha.
\]
For small $\alpha$ this gives
\[
P(x,y) = \cos\alpha\,\sin\alpha
= \left(1-\frac{\alpha^2}{2}+O(\alpha^4)\right)\left(\alpha-\frac{\alpha^3}{6}+O(\alpha^5)\right)
= \alpha - \frac{2}{3}\alpha^3 + O(\alpha^5).
\]

The same expansion applies to the pairs $(x+y,y)$ and $(x,x+y)$ occurring in $B_n$, with the corresponding angles at the origin. The sum of linear terms $\alpha$ in the expansion of $S_n$ tends to
$ -\pi/2$, while the sum of the $O(\alpha^3)$
terms converges to zero, so the total boundary contribution tends to $-\pi/2$, i.e.\ we obtain
\[
P\left(e_1,e_2\right) + S_n \longrightarrow -\frac{\pi}{2}
\]
as $n\to\infty$, because $P(e_1,e_2)$ is zero. 

On the other hand, applying \eqref{eq_1} with $\det(x,y)=1$ for $(x,y)\in A_n$ gives
\[
\sum_{(x,y)\in A_n} \bigl(P(x,y)-P(x+y,y)-P(x,x+y)\bigr)
= -2\sum_{(x,y)\in A_n} \frac{1}{\|x\|^2\,\|y\|^2\,\|x+y\|^2}.
\]
Passing to the limit $n\to\infty$ and comparing with the previous expression, we obtain
\[
-2\sum_{(x,y)\in SL_+(2,\ZZ)} \frac{1}{\|x\|^2\,\|y\|^2\,\|x+y\|^2}
= -\frac{\pi}{2},
\]
so
\[
\sum_{(x,y)\in SL_+(2,\ZZ)} \frac{1}{\|x\|^2\,\|y\|^2\,\|x+y\|^2}
= \frac{\pi}{4},
\]
which is equivalent to the stated identity.
\end{proof}

Note that, to obtain telescoping relations such as \eqref{eq_1}, one may consider other quadratic forms in place of $q(x)=\|x\|^2$. This leads to telescoping identities over topographs; see the separate article \cite{kalinintopo}. There are several stories here, all based on telescoping but with very different flavors, so it is difficult to place them under a single unifying framework.

\begin{theorem}[\cite{easter}]
\label{thm_old}
 $$\sum\limits_{(x,y)\in SL_+(2,\ZZ)} (\|x\|\,+\|y\|\,-\|x+y\|)= 2.$$   
 $$\sum\limits_{(x,y)\in SL_+(2,\ZZ)} (\|x\|\,+\|y\|\,-\|x+y\|)^2 = 2-\pi/2.$$ 
\end{theorem}
\begin{proof}
We provide a telescopic proof of these identities. Recall that
\[
F(x,y)=2\bigl(G(x,y)+G(y,x)\bigr)
=2\left(\|x\|-\frac{x\cdot y}{\|y\|}+\|y\|-\frac{x\cdot y}{\|x\|}\right)
\]
is the function $F$ from Theorem~\ref{thm_xyz}. A direct computation shows the exact identity
\[
F(x, y) - F(x + y, y) - F(x, x + y)
= 2\bigl(\|x\|+\|y\|-\|x+y\|\bigr).
\]
By the telescoping argument over $SL_+(2,\ZZ)$, this proves the first equality.

The second equality can also be proved by telescoping\footnote{\url{https://mathoverflow.net/questions/250041/the-number-pi-and-summation-by-sl2-mathbb-z}} (and by different methods \cite{easter}). Consider the function $T$ defined on pairs of vectors by
\[
T(x,y) = 2\bigl(\|x\|\,\|y\|-x\cdot y\bigr).
\]
One checks that
\[
T(x,y)-T(x+y,y)-T(x,x+y)
= \bigl(\|x\|+\|y\|-\|x+y\|\bigr)^2.
\]
Applying telescoping, we cancel all but boundary terms. The boundary at the root is $T(e_1,e_2)=2$, while the boundary terms at infinity satisfy
\[
T(x,y)= 2\,\frac{1-\cos\alpha}{\sin\alpha}+O(\alpha^3) = \alpha+O(\alpha^3),
\]
and the sum of the corresponding angles is $\pi/2$. This gives a total contribution of $\pi/2$ from infinity, and hence the second identity.
\end{proof}

\begin{theorem}\label{th4}
\[
\sum_{(x,y)\in SL_+(2,\ZZ)} 
\frac{\|x\|+\|y\|-\|x+y\|}{\|x\|\,\|y\|\,\|x+y\|}
= \frac{\pi}{2}-1.
\]
\end{theorem}

\begin{proof} Define
\[
K(x,y)=\frac{1}{\|x\|\,\|y\|}.
\]
Then
\[
K(x,y)-K(x+y,y)-K(x,x+y)
= \frac{\|x+y\|-\|x\|-\|y\|}{\|x\|\,\|y\|\,\|x+y\|}.
\]
Telescoping over $SL_+(2,\ZZ)$ shows that the sum of the right-hand side equals
\[
K(e_1,e_2)-\lim_{\text{boundary}}\sum K(x,y).
\]
As in the proof of Theorem~\ref{thm_old}, the boundary terms at infinity behave like
\[
K(x,y)=\alpha +O(\alpha^3)
\]
where $\alpha$ is the angle between $x$ and $y$; and their total contribution is $\pi/2$  (the same angle-partition argument as in the proof of Theorem~\ref{th_one}), while $K(e_1,e_2)=1$. This yields the stated value $\pi/2-1$ when we proceed as in Theorem~\ref{th_one}.
\end{proof}

\section{Derivation of Zagier's lattice sum formula}\label{sec_4}

We now present an alternative proof of Zagier's formula using our methods. We split the sum into collinear and non-collinear triples. Collinear triples give the Eisenstein part $2E(i,3)$; non-collinear triples are counted modulo central symmetry, which contributes the $\pi^3\zeta(3)$ term. We now treat these two parts.
Let 
\[
H = \bigl(\ZZ\times\ZZ_{>0}\bigr)\,\cup\,\bigl(\ZZ_{\ge0}\times\{0\}\bigr),
\]
i.e.\ the set of lattice vectors in the closed upper half-plane with the negative $x$-axis removed.

\begin{theorem}
\label{th_three}
For every integer $n\ge1$,
\[
\sum_{\substack{x,y\in H\\ \det(x,y)=n}}
 \frac{n^2}{\|x\|^2\,\|y\|^2 \,\|x+y\|^2}
 = \frac{\pi}{2n}\,\sigma_1(n).
\]
\end{theorem}

\begin{proof}

We use \eqref{eq_1} and follow the proof of Theorem~\ref{th_one}. Note that there are $\sigma_1(n)$ (the sum of divisors of $n$) inequivalent sublattices of $\ZZ^2$ with determinant $n$, each of which is generated by a basis of the form 
\[
\left(\tfrac{n}{d},0\right),\ (k,d),\quad 0\leq k<d,
\]
where $d$ is a divisor of $n$. 
 For a fixed sublattice $L$ of this type (with given $k$ and $d$), the vectors
\[
(k+j\tfrac{n}{d},d),\qquad j\in\ZZ,
\]
partition the angle $\pi$ of the upper half-plane as seen from the origin (cf.\ the angle $\pi/2$ in Theorem~\ref{th_one}). For $\det(x,y)=n$ we have
\[
n = \|x\|\,\|y\|\,\sin\alpha,
\]
hence, using the same function $P$ from the proof of Theorem~\ref{th_one},
\[
P(x,y)
= \frac{\cos\alpha\sin\alpha}{n}
= \frac{1}{n}\,\bigl(\alpha+O(\alpha^3)\bigr),
\]
where $\alpha$ is the angle between $x$ and $y$. Thus, the boundary contribution in the telescoping sum associated to \eqref{eq_1} tends to $-\pi/n$, while the root contribution $P\left(\vect{0}{\tfrac{n}{d}},\vect{k-N\tfrac{n}{d}}{d}\right)\to 0$ when $N\to\infty$. Using \eqref{eq_1} with $\det(x,y)=n$ gives, for each such lattice $L$,
\[
-2n^2 \sum_{\substack{x,y\in H\cap L\\ \det(x,y)=n}}
\frac{1}{\|x\|^2\,\|y\|^2\,\|x+y\|^2}
= -\frac{\pi}{n},
\]
hence
\[
\sum_{\substack{x,y\in H\cap L\\ \det(x,y)=n}}
\frac{n^2}{\|x\|^2\,\|y\|^2\,\|x+y\|^2}
= \frac{\pi}{2n}.
\]
Summing over all $\sigma_1(n)$ index-$n$ sublattices completes the proof.
\end{proof}

Look at
\[
D_{1,1,1}(i)
= \sideset{}{'}\sum_{\omega_1+\omega_2+\omega_3=0}
\frac{1}{|\omega_1\omega_2\omega_3|^{2}}, \quad \omega_k\in \ZZ i+\ZZ.
\]

Among such triples there are collinear triples. If $\omega_1,\omega_2,\omega_3$ are collinear, let $\omega$ be the primitive vector in the common direction of the two vectors pointing the same way. Then the triple is of the form
\[
(b\omega,\ d\omega,\ -(b+d)\omega),
\qquad b,d\in\ZZ_{>0},
\]
and any of the three vectors can be the opposite of the other two, yielding a factor three. Using the classical Tornheim series identity \cite{tornheim1950harmonic}
\[
\sum_{b,d\in \ZZ_{>0}}\frac{1}{b^2d^2(b+d)^2} = \frac{\zeta(6)}{3},
\]
the contribution of all collinear triples is
\begin{align*}
\sum_{\substack{\omega\ \text{primitive}}}
\frac{1}{|\omega|^6}\cdot 3\sum_{b,d>0}\frac{1}{b^2d^2(b+d)^2}
&=
\sum_{\substack{\omega\ \text{primitive}}}
\frac{\zeta(6)}{|\omega|^6}.
\end{align*}
Since
\[
\sum_{\omega\in\ZZ^2\setminus\{0\}}\frac{1}{|\omega|^6}
= \zeta(6)\sum_{\substack{\omega\ \text{primitive}}}\frac{1}{|\omega|^6}
= 2E(i,3),
\]
we obtain
\[
D_{1,1,1}(i)\big|_{\text{collinear}} = 2E(i,3),
\]
as in the first term in~\eqref{eq_2}.

For non-collinear triples, up to central symmetry we may assume that two of the vectors (say, $\omega_1,\omega_2$) lie in $H$, and that the parallelogram spanned by $\omega_1$ and $\omega_2$ has positive signed area. Each non-collinear triple yields exactly $12$ such ordered pairs $(x,y)$ with $x,y\in H$ and $\det(x,y)>0$ (six permutations of $(\omega_1,\omega_2,\omega_3)$ and a factor $2$ from central symmetry). Thus
\[
D_{1,1,1}(i) -  2 E(i,3)
= 12 \sum_{n\in\ZZ_{>0}} \sum_{\substack{x,y\in H\\ \det(x,y)=n}} \frac{1}{\|x\|^2\,\|y\|^2\,\|x+y\|^2}.
\]
By Theorem~\ref{th_three}, for each $n\ge1$ we have
\[
\sum_{\substack{x,y\in H\\ \det(x,y)=n}}
\frac{n^2}{\|x\|^2\,\|y\|^2\,\|x+y\|^2}
= \frac{\pi}{2n}\,\sigma_1(n),
\]
hence
\[
\sum_{\substack{x,y\in H\\ \det(x,y)=n}}
\frac{1}{\|x\|^2\,\|y\|^2\,\|x+y\|^2}
= \frac{\pi}{2n^3}\,\sigma_1(n).
\]
Therefore
\begin{align*}
\pi^3\zeta(3)
&= D_{1,1,1}(i) -  2 E(i,3) \\
&= 12 \sum_{n\ge1} \sum_{\substack{x,y\in H\\ \det(x,y)=n}} \frac{1}{\|x\|^2\,\|y\|^2\,\|x+y\|^2} \\
&= 12 \sum_{n\ge1} \frac{\pi}{2n^3}\,\sigma_1(n)
= 6\pi \sum_{n\ge1} \frac{\sigma_1(n)}{n^3}
= 6\pi\,\zeta(3)\zeta(2)
= \pi^3\zeta(3),
\end{align*}
since $\sum_{n\ge1} \sigma_1(n)n^{-s} = \zeta(s)\zeta(s-1)$ and $\zeta(2) = \pi^2/6$. We have thus recovered Zagier's formula \eqref{eq_2} for $z=i, s=3$.

Similarly, including the determinant weight and repeating the above argument yields
\[
\sideset{}{'}\sum_{\substack{\omega_1+\omega_2+\omega_3=0\\ \omega_k\in \ZZ i+\ZZ}}
\frac{|\det (\omega_1,  \omega_2)|^{-s}}{|\omega_1\omega_2\omega_3|^{2}}
= 6\pi\,\zeta(s+3)\,\zeta(s+2),
\]
where the prime indicates that collinear triples (for which $\det(\omega_1,\omega_2)=0$) are omitted.

By replacing the lattice $\ZZ i+\ZZ$ with an arbitrary lattice $\ZZ z+\ZZ$, $z=x+iy\in\mathcal H$, we obtain:

\begin{theorem}
For $z=x+iy\in\mathcal H$ and $s>-1$,
\[
\sideset{}{'}\sum_{\substack{\omega_1+\omega_2+\omega_3=0\\ \omega_k\in \ZZ z+\ZZ}}
\frac{|\det (\omega_1,  \omega_2)|^{-s}}{|\omega_1\omega_2\omega_3|^{2}}
= \frac{6\pi}{y^{3}}\,\zeta(s+3)\,\zeta(s+2).
\]
\end{theorem}

Setting $s=0$ in the theorem above gives the contribution of non-collinear triples to $D_{1,1,1}(z)$, while the collinear contribution equals $2E(z,3)$ by the same argument as for $z=i$. Combining these two pieces yields Zagier's formula \eqref{eq_2} for an arbitrary $z$:
\[
D_{1,1,1}(z) = 2E(z,3)+\pi^3\zeta(3).
\]

\section{Convergence of  $\sideset{}{'}\sum \frac{\det(x,y)^{-s}}{\|x\|\,\|y\|\,\|x+y\|}$}\label{sec_5}

In this section we use the notation of the preceding sections.

\begin{theorem}
\label{thm_s}
\begin{equation}
\label{sum_1}
S^{\det}(s,z)
=
\sideset{}{'}\sum_{\substack{w_1 + w_2 + w_3 = 0 \\ w_i \in \mathbb{Z} + \mathbb{Z} z}}
 \frac{|\det(w_1, w_2)|^{-s}}{|w_1|\,|w_2|\,|w_3|}
\end{equation}
is a real-analytic automorphic function of $z \in \mathcal{H}
:= \{ z \in \mathbb{C} : \Im > 0 \}$, defined for $s>0$, in the sense
that for $s>0$ the series \eqref{sum_1} converges absolutely.
\end{theorem}
Here the prime on the sum indicates that terms with $w_i=0$ or
$\det(w_1,w_2)=0$ are omitted.

First, note that $S^{\det}(s,z)$ transforms like a real-analytic modular form
of weight $-(2s+3)$ under $\mathrm{SL}_2(\ZZ)$ (this follows by a change of lattice
basis). To prove convergence we need the following lemma. Recall that
\[
H = \bigl(\ZZ\times\ZZ_{>0}\bigr)\,\cup\,\bigl(\ZZ_{\ge0}\times\{0\}\bigr),
\]
that is, the set of lattice vectors in the closed upper half-plane with the
negative $x$-axis removed.

\begin{lemma}
\label{lem_sum3}
\begin{equation}
\label{sum_3}
\sum_{\substack{x,y \in H \\ \det(x,y) = n}} \frac{1}{\|x\|\,\|y\|\,\|x+y\|}
= O\!\left(\frac{\log(n)\sigma_0(n)}{n}\right),
\end{equation}
where the implied constant is absolute (independent of $n$).
\end{lemma}

\begin{proof}
To prove Lemma~\ref{lem_sum3} we use the following description of all lattices in $\mathbb{Z}^2$ of determinant $n$.

Each such lattice is generated by two vectors $(d,0), \left(k, \frac{n}{d} \right)$ where $d|n$ and $0 \le k \le d-1$. Since the number of possible $d$ is $\sigma_0(n)$, our goal is to show that lattices of determinant $n$ with given $d$ (the sum over  $0 \le k \le d-1$) contribute $O(\frac{\log(n)}{n})$.

Say that a pair of vectors $x_0, y_0$ with $\det(x_0,y_0)=n$  {\it generates} pairs given by
\[
\begin{pmatrix}
a & b \\
c & d
\end{pmatrix}
x_0, \quad
\begin{pmatrix}
a & b \\
c & d
\end{pmatrix}
y_0, \quad \text{where } ad - bc = 1, \, a,b,c,d \geq 0,
\]
i.e.\ all the pairs $(x,y)\in H, \det(x,y)=n$ that lie in the {\it cone} generated by $x_0$ and $y_0$. We denote such a cone by  $\langle x_0,y_0\rangle $.

What follows is, essentially, the same as the proof of Theorem~\ref{th_three}, but the final computation is a bit more involved. To generate all pairs $(x,y)\in H, \det(x,y)=n$ one should consider all divisors $d$ of $n$ and for each $d$ for all $0\leq k<d$ take all  the cones $$\langle (d,0), \left(k, \tfrac{n}{d}\right)\rangle , \langle \left(k-pd, \tfrac{n}{d}\right), \left(k-(p+1)d, \tfrac{n}{d}\right)\rangle ,\ p\geq 0.$$
Those with $p\geq 1$ are symmetric to the cones inside $\langle (d,0), \left(k, \frac{n}{d}\right)\rangle $. Thus, to prove the lemma, it suffices to consider the pairs $(x,y)\in H, \det(x,y)=n$ in cones 
$$C_k=\langle (d,0), \left(k, \tfrac{n}{d}\right)\rangle , C'_k=\langle \left(k-d, \tfrac{n}{d}\right), \left(k, \tfrac{n}{d}\right)\rangle  \text{ for all   }1\leq k \le d-1,$$
and sum $\frac{1}{\|x\|\,\|y\|\,\|x+y\|}$ over pairs $(x,y)=\left((d,0),(k-pd, \frac{n}{d})\right)$ for $p\in\ZZ_{\geq 0}$. 

For the same $F(x,y)$ as in Theorem~\ref{thm_xyz} and $x,y$ as above
\[
F(x,y) = 2(\|x\|\, + \|y\|\,)\left(1 - \frac{x \cdot y}{\|x\|\,\|y\|\,}\right)
= \frac{2(\|x\|\,+\|y\|\,)n^2}{\|x\|\,\|y\|\,(\|x\|\,\|y\|\,+x\cdot y)},\]
we have
\[
F(x,y) - F(x,x+y) - F(x+y,y) = \frac{n^2(1+O(\alpha))}{\|x\|\,\|y\|\,\|x+y\|}.
\]

Since in the cones $C_k$ we have $0\leq x\cdot y\leq \|x\|\,\|y\|$ (we call such cones {\it thin}), to estimate \eqref{sum_3} we should sum $\frac{1}{\|x\|\,\|y\|\,\|x+y\|}$ over all pairs $(x,y)=\left((d,0),(k-pd, \frac{n}{d})\right)$ and add the sum of
$$\frac{\|x\|+\|y\|}{\|x\|^2\|y\|^2} \leq \frac{F(x,y)}{n^2}\leq 2\frac{\|x\|+\|y\|}{\|x\|^2\|y\|^2}$$
over cones $A_k$.  We compute contribution of $C'_k$ a bit later (also subdividing it in thin cones with $0\leq x\cdot y$). So, consider
\[
2\left(\sum_{k=1}^{d-1} \frac{\left(d + \sqrt{k^2 + \frac{n^2}{d^2}}\right)}{d^2 \left(k^2 + \frac{n^2}{d^2}\right)}
+ \frac{\left(d + \frac{n}{d}\right)}{d \cdot \frac{n}{d} \left(d \cdot \frac{n}{d} + 0\right)}\right)+
\]
\[
+\sum_{k=1}^{d-1} \frac{1}{\left(\sqrt{k^2 + \frac{n^2}{d^2}}\sqrt{(k-d)^2 + \frac{n^2}{d^2}}\right) \cdot d}
%+ \sum_{k=1}^{d-1} \frac{2\left(\sqrt{k^2 + \frac{n^2}{d^2}} + \sqrt{(k-d)^2 + \frac{n^2}{d^2}}\right)}{\left(k^2 + \frac{n^2}{d^2}\right)\left((k-d)^2 + \frac{n^2}{d^2}\right)}
\]

Up to multiplication by a constant, this is
\begin{align*}
\frac{2\left(d + \frac{n}{d}\right)}{n^2}
+ \frac{2}{d}\int_0^d \frac{1}{\left(x^2 + \frac{n^2}{d^2}\right)}dx
+\frac{2}{d^2}\int_0^d \frac{1}{\sqrt{x^2 + \frac{n^2}{d^2}}}dx+\\
+\frac{1}{d}\int_0^d  \frac{1}{\left(\sqrt{x^2 + \frac{n^2}{d^2}}\sqrt{(x-d)^2 + \frac{n^2}{d^2}}\right)}dx
%+ \int_0^d \frac{1}{\left(\sqrt{x^2 + \frac{n^2}{d^2}}\right)\left((x-d)^2 + \frac{n^2}{d^2}\right)}dx
\end{align*}

Now,

\[
\frac{1}{d}\int_0^d \frac{1}{x^2 + \frac{n^2}{d^2}}dx = \frac{1}{n}\arctan(d^2/n)\leq \frac{\pi/2}{n}
\]

\[
\frac{1}{d^2}\int_0^d \frac{1}{\sqrt{x^2 + \frac{n^2}{d^2}}}dx = \frac{1}{d^2}\operatorname{arcsinh}(d^2/n)\leq  \frac{C}{n}
\]
\[
\frac{1}{d}\int_0^d  \frac{1}{\left(\sqrt{x^2 + \frac{n^2}{d^2}}\sqrt{(x-d)^2 + \frac{n^2}{d^2}}\right)}dx
 \leq \frac{C}{n}.
\]

%Finally, plugging $d=n$ into
%
%$$\sum_{k=1}^{d-1} \frac{2\left(\sqrt{k^2 + \frac{n^2}{d^2}} + \sqrt{(k-d)^2 + \frac{n^2}{d^2}}\right)}{\left(k^2 + \frac{n^2}{d^2}\right)\left((k-d)^2 + \frac{n^2}{d^2}\right)}$$
%
%we get 
%
%$$\sum_{k=1}^{n-1} \frac{2\left(\sqrt{k^2 + 1} + \sqrt{(k-n)^2 + 1}\right)}{\left(k^2 + 1\right)\left((k-n)^2 + 1\right)}\approx \frac{n}{n^3}\int_{1/n}^{{n-1}/n}\frac{1}{x^2(1-x)^2}dx\approx \frac{n}{n^3}n=\frac{1}{n}.$$
%
%Moreover, for other values of $d$ the corresponding sum is even smaller. 

The remaining contribution of cones $C'_k$ may be computed in the following way. Consider the cone $C'_{k}=\langle \left(k-d, \tfrac{n}{d}\right), \left(k, \tfrac{n}{d}\right)\rangle$ for $d=n$. For $j=d-k$, the sum of $\frac{1}{\|x\|\,\|y\|\,\|x+y\|}$ over $(x,y_i)=\left((-j,1),(n-i\cdot j,1+i)\right), i=1,2,\dots,[\frac{n}{{j+1}}]=m$ is of order $\tfrac{1}{(j+1)n}$ because
$$\frac{1}{2}\cdot \frac{1}{(n-i\cdot j)\bigl(n-(i+1)j\bigr)}\leq \frac{1}{\|y_i\|\|x+y_i\|}\leq \frac{1}{(n-i\cdot j)\bigl(n-(i+1)j\bigr)},$$

and 

\[
\frac{1}{\|x\|}\sum_{i=1}^{m}
\frac{1}{(n-i\cdot j)\bigl(n-(i+1)j\bigr)}
=
\frac{1}{\|x\|}\frac{1}{j}
\left(
\frac{1}{n-(m+1)j} - \frac{1}{n-j}
\right) = \frac{1}{\|x\|}O\left(\frac{1}{n}\right) 
\]
So the sum of contributions is of order $\tfrac{1}{n}(1+\tfrac{1}{2}+\tfrac{1}{3}+\dots +\tfrac{1}{n})=\log(n)\tfrac{1}{n}$.

For the pairs $(x,y)=\left((-j,\tfrac{n}{d}),(d-i\cdot j,(1+i)\tfrac{n}{d})\right), i=1,2,\dots,[\frac{d}{{j+1}}]$ the contribution $d<n$ is bounded above by the corresponding contribution for $d=n$. Similarly, the contribution is of order $\frac{\log(n)}{n}$ when we sum $\frac{\|x\|\,+\|y\|\,}{\|x\|^2\|y\|^2}$ over the remaining thin cones
 $$\langle \left(-j, \tfrac{n}{d}\right), \left(d-[\tfrac{d}{j+1}](j+1), \tfrac{n}{d}(1+[\tfrac{d}{j+1}])\right)\rangle.$$

Therefore, summing over all divisors $d|n$ we obtain
\[
\sum_{\substack{x,y \in H \\ \det(x,y) = n}} \frac{1}{\|x\|\,\|y\|\,\|x+y\|}
= O\!\left(\frac{\log(n)\sigma_0(n)}{n}\right).
\]
%since $\frac{\arctan(d^2/n)}{n}=O(1/n)$ uniformly for all $d\ge \sqrt{n}$.

\end{proof}
\begin{theorem}\label{thm_s2} 
Let \( H \) be as in Section~\ref{sec_4}. Then
\[
\sum_{\substack{x,y \in H \\ \det(x,y)\neq 0}}
 \frac{|\det(x,y)|^{-s}}{\|x\|\,\|y\|\,\|x+y\|}
\quad\text{converges for all } s>0.
\]\end{theorem}

\begin{proof}
By Lemma~\ref{lem_sum3}, the contribution of all pairs with $\det(x,y)=n$
is $O\bigl(\frac{\log(n)\sigma_0(n)}{n}\bigr)$. Multiplying by $|\det(x,y)|^{-s}=n^{-s}$,
we obtain a total contribution
\[
\sum_{n\ge1} \frac{
\log(n)\sigma_0(n)}{n^{1+s}},
\]
which converges for all $s>0$, as claimed. 
\end{proof}

\begin{proof}[Proof of Theorem~\ref{thm_s}]
Choose $A\in\mathrm{GL}_2(\RR)$ such that $A(\ZZ z+\ZZ)=\ZZ i+\ZZ$. Since $A$ is invertible and linear, it distorts Euclidean norms by at most a multiplicative constant, and $\det(Aw_1,Aw_2)=(\det A)\,\det(w_1,w_2)$. Thus each summand in \eqref{sum_1} for the lattice $\ZZ z+\ZZ$ is comparable (up to a constant factor depending only on $A$) to a summand in the corresponding series for the lattice $\ZZ i+\ZZ$. Hence absolute convergence of $S^{\det}(s,i)$ implies absolute convergence of $S^{\det}(s,z)$, and the result follows from Theorem~\ref{thm_s2}.
\end{proof}

\section{Acknowledgements}

I would like to thank Mikhail Shkolnikov and Ernesto Lupercio for fruitful discussions and an anonymous referee for the comments which improved this paper.

%\bibliography{../bibliography.bib}
%\bibliographystyle{abbrv}

\end{document}